\documentclass[final,3p,times]{elsarticle}

\usepackage{graphicx}
\usepackage{amssymb}
\usepackage{amsmath}
\usepackage[dvipsnames]{xcolor}
\usepackage{commath}
\usepackage{amsfonts}%
\usepackage{graphicx,verbatim}
\usepackage{tikz}
\usetikzlibrary{cd}

\usepackage[pagebackref=false,colorlinks=true, pdfstartview=FitV, linkcolor=violet,citecolor=red, urlcolor=blue]{hyperref}

\usepackage{amsthm}
\newtheorem{lthm}{Theorem}

\usepackage{units}

\newtheorem{theorem}{Theorem}
\newtheorem*{theorem*}{Theorem}
\newtheorem{definition}{Definition}
\newtheorem{proposition}{Proposition}
\newtheorem{corollary}{Corollary}
\newtheorem{lemma}{Lemma}
\newtheorem{example}{Example}
\newtheorem{conjecture}{Conjecture}
\newtheorem*{conjecture*}{Conjecture}

\usepackage{lineno}

\journal{arXiV}

\begin{document}

\begin{frontmatter}


\title{Knots inside Fractals}



\author{Joshua Broden}
\ead{jbbroden@uwaterloo.ca, University of Waterloo}

\author{Malors Espinosa}
\ead{srolam.espinosalara@mail.utoronto.ca, Department of Mathematics, University of Toronto}

\author{Noah Nazareth}
\ead{ncnazare@uwaterloo.ca, University of Waterloo}

\author{Niko Voth}


\begin{abstract}
We prove that all knots can be embedded into the Menger Sponge fractal. We prove that all Pretzel knots can be embedded into the Sierpinski Tetrahedron. Then we compare the number of iterations of each of these fractals needed to produce a given knot as a mean to compare the complexity of the two fractals.
Word count: 4775

\end{abstract}
\end{frontmatter}


\section{Introduction and Motivation}\label{sec: introduction}

\subsection{The questions of this paper}

The Menger Sponge is a fractal obtained by iteratively subdividing a cube into $27$ equal cubes and removing the central cube of each face and the interior central cube.  In 1926, Menger proved  the following 

\begin{theorem*}[Menger, 1926]
    The Menger Sponge is universal for all compact one dimensional topological spaces.
\end{theorem*}

The concept of \textit{universality} in the previous theorem means that any compact topological space of dimension $1$ has a homeomorphic copy as a subspace of the Menger Sponge. We refer the reader to \cite[Chapter 2]{chaosandfractals} for a discussion on these topics. 

A particular example of such space is the circle $\mathbb{S}^1$. By inspecting the Menger Sponge, or its first iterations, we can see that the circle can be found as a subspace in many different ways, some rather complicated. In particular, some of them are actually knotted. This is what motivates our two questions:

\begin{description}
    \item[Question 1:] Can we find every knot as a subspace of the Menger Sponge? More particularly, can we find every knot in the $1$-skeleton (that is, the edges of the boundary) of a finite iteration of the Menger Sponge?
    
    By this we mean that we do not only want $\mathbb{S}^1$ embedded in any which way but in a given way that is isotopic to a given knot and created by the edges of the boundary of a finite iteration. The reason for this choice is that, otherwise, we cannot guarantee points on it do not get removed at further stages of the iterative process.

    \item[Question 2:] Can we find fractals, also created by iterative processes, where we can also find all the knots in their finite iterations? We would expect that more complicated fractals produce complicated knots faster than simpler fractals.
\end{description}
Notice that question 1 does not follow from universality, since all knots are homeomorphic to $\mathbb{S}^1$. Thus, the embedding in the Menger Sponge might not preserve the knot type!  

\subsection{Results of this paper}

The results we will prove on this paper, with regard to the two questions posed above are as follows. We completely settle the first question:

\begin{lthm}
    Any knot $K$ can be embedded in a finite iteration of the Menger Sponge.
\end{lthm}

To prove the above result we will use the Arc Presentation (in grid form) of $K$ and its associated connectivity graph. Concretely, we will find the connectivity graph on the face of a Menger Sponge and then prove that we can resolve each intersection, by pushing it into the sponge in such a way that we avoid all the holes.  As a consequence of the proof of the above result we will deduce the well known fact that the isotopy classes of tame knots form a countable set.

To discuss question 2 we explore another fractal, the Sierpinski Tetrahedron. It is also created by an iterative process, which we describe below in section \ref{sec: review of prereq}. The strategy used before, where we found the connectivity graph on a face and then push into the fractal, doesn't work as there seem to be no appropriate knot diagram suitable for this.

Instead, we introduce the combinatorial diagram of the tetrahedron, which is a  two dimensional map of the fractal that allow us to look either by inspection or by structure of certain families of knots, in the map and transfer into the tetrahedron. We prove

\begin{lthm}
   Any pretzel knot $K$ can be embedded in a finite iteration of the Sierpinski Tetrahedron.
\end{lthm}

However, we were unable to prove that in general any knot is inside the tetrahedron, but we do not see any real impediment to find them. Thus we leave open the following

\begin{conjecture*}
    Every knot $K$ can be embedded into the Sierpinski Tetrahedron.
\end{conjecture*}

The Sierpinski Tetrahedron is \textit{simpler} than the Menger Sponge. We use these two fractals to study the second question. To do so we define $M(K)$ as the minimal number of iterations of the Menger Sponge needed for $K$ to be embedded in it. For example, we have $M(3_1) = 1$. Similarly, whenever a knot $K$ is embedded in a finite iteration of the Sierpinski Tetrahedron, we denote the minimal such iteration by $S(K)$.
With this notation at hand, we prove
\begin{lthm}
Let $K$ be a knot, then we have
\begin{equation*}
    M(K) \le S(K),
\end{equation*}
supposing $S(K)$ is defined.
\end{lthm}
 We see the above result as a positive answer to the second question posed above: indeed, when using knots as probes to measure the complexity of a fractal, the simpler fractal take longer to include more complicated knots.

\subsection{Previous Literature}

Knots and graph theory have been related for a long time. Frequently, constructions of graph theory are used to compute and create knot invariants. We refer the reader to \cite{KnotsnGraphs}, and the bibliography therein, for examples and the history of this connection.

A particular set of graphs which will be related to our combinatorial representation appears in \cite{SierpinskiGraphs}, and is denoted $S_4^n$ in there. The main difference is that for us the vertices are resolved, that is, we know which edge goes above and which one goes below. This difference is fundamental since we use the knot diagrams found in these combinatorial representations to construct the actual knots we are looking for. To be more precise, if we were to find the connectivity graphs of knots in these $S_4^n$, we would still need to know that the vertices can be resolved consistently with the three dimensionality of the tetrahedron.

For instance, in our example \ref{ex: example 3_1 sierpinski} below, we find the trefoil in an iteration of the Sierpinski Tetrahedron. If we change all crossings to a vertex we obtain the graph $S_4^4$ (see \cite[Figure 7, page580]{SierpinskiGraphs}). However, we need to know which line goes above and below to obtain a real knot diagram of the trefoil.

Finally, when one searches for knots and fractals, not much appears. Again, some connections exists aimed at the computation of knot invariants, as for example in \cite{KnotPartitionFractals}. Other directions, frequently aimed at artistic renditions and the search for beautiful patterns in them, deal with knots obtained by iteratively substituting crossings by smaller and smaller versions of the original knot. This leads to certain wild knots that resemble fractals.

In this paper, however, our point of view has been to use knots and their complexity to study how different fractals, created by iterative processes, manifest their complexity. Our guiding principle has been the idea that more complicated fractals should produce complicated knots \textit{faster}. It seems this perspective has not been explored before.

\subsection{Acknowledgements}

The authors thank the Outreach Department of the Department of Mathematics of the University of Toronto for the organization of the mentorship programs and their support during this time. We also wish to thank the Department of Mathematics of the university of Toronto, as well as the organizing committee of the Canadian Undergraduate Mathematics Conference 2022, for their support in the presentation of this work in the CUMC 2022. 

The three dimensional pictures that appear in this work have been created with Blender. The Arc Presentations in grid form of $3_1$ and $4_1$ were taken from The Knot Atlas.
\section{Review of prerequisites}\label{sec: review of prereq}

\subsection{The Cantor Set, Cantor Dust, Sierpinski Carpet, Sierpinski Tetrahedron and the Menger Sponge.}

Let $I = [0, 1]$. By the Cantor set we mean the classical \textit{one third Cantor set}. We denote it by $C$ and recall the following classical characterization:

\begin{proposition}\label{prop: cantor set characterization}
A number $x\in [0, 1]$ is a \textit{Cantor number}, that is $x\in C$, if and only if in its ternary non-ending representation there are no digit $1$'s appearing.
\end{proposition}
\begin{proof}
    See \cite[page 72]{chaosandfractals}.
\end{proof}

The product of two Cantor Sets, $C\times C$, is called \textit{Cantor Dust}. It can also be constructed in an iterative process as follows:
\begin{itemize}
    \item Define $C_0 = I\times I$.
    \item Divide $C_0$ into nine equal squares of length size $1/3$. To obtain $C_1$, remove any point that is contained only in the central square or the middle central squares of the sides. Notice that $C_1$ is closed.
    \item Iterate this process for each remaining square of $C_n$ to obtain $C_{n + 1}$.
\end{itemize}
We then have that
\begin{equation*}
    C\times C = \displaystyle\bigcap_{n = 0}^{\infty} C_n.
\end{equation*}
We also can characterize Cantor Dust in terms of ternary representations. More precisely we have

\begin{proposition}\label{prop: cantor dust characterization}
A point $(x, y)\in I\times I$ is a Cantor Dust Point, that is $(x, y)\in C\times C$ if and only if both $x$ and $y$ do not have a digit one in their corresponding ternary non-ending representations.
\end{proposition}
\begin{proof}
    This follows immediately from proposition \ref{prop: cantor set characterization}.
\end{proof}

We now discuss is the Sierpinski Carpet (See \cite[section 2.2]{chaosandfractals}). It is also defined by an iterative process as follows:
\begin{itemize}
    \item Define $s_0 = I\times I$.
    
    \item Divide $s_0$ into $9$ equal squares of side length 1/3. Remove the center open one.
    
    \item Repeat the previous step, removing the central square, to each one of the remaining squares of $s_n$ to obtain $s_{n+1}$.
\end{itemize}
The \textit{Sierpinski Carpet} is defined as
\begin{equation*}
    s = \displaystyle\bigcap_{n = 0}^{\infty}s_n.
\end{equation*}

Just as the Cantor set and the Cantor Dust, we have a characterization of $s$ in terms of ternary representations. Concretely, we have
\begin{proposition}\label{prop: sierpinski carpet characterization}
A point $(x, y)\in I\times I$ is a Sierpinski Carpet point, that is $(x, y)\in s$ if and only if $x$ and $y$ do not have a digit one in the same position in their nonending ternary representation.
\end{proposition}
\begin{proof}
    See \cite[page 405, Example 14.1.1]{Allouche_Shallit_2003}.
\end{proof}

An immediate conclusion from propositions \ref{prop: cantor dust characterization} and \ref{prop: sierpinski carpet characterization} is the following

\begin{corollary}\label{cor: cantor dust subset sierpinski carpet}
The Cantor Dust is a subset of the Sierpinski Carpet.
\end{corollary}
\begin{proof}
The points of the Cantor Dust have no ones in the ternary representations of their coordinates, so they cannot share a digit one in the same position.
\end{proof}

Now we discuss the Sierpinski Tetrahedron. It is also defined by an iterative process as follows:

\begin{itemize}
    \item Define $t_0$ as a closed regular tetrahedron.
    
    \item On each corner, keep the homothetic closed tetrahedron with length size scaled by $1/2$. Remove the rest to obtain $t_1$.
    
    \item Repeat the previous step on each remaining tetrahedron of the previous step to obtain $t_{n+1}$ from $t_n$.
\end{itemize}
The \textit{Sierpinski Tetrahedron} is defined as
\begin{equation*}
    T = \displaystyle\bigcap_{n = 0}^{\infty}t_n.
\end{equation*}

The final fractal we will need is produced iteratively as follows:
\begin{itemize}
    \item Define $M_0 = I\times I\times I$.
    \item Divide $M_0$ into 27 equal cubes of side-length size $1/3$. To obtain $M_1$, remove any point that is only contained in the central cube of each face or the open central cube of $M_0$. Notice that $M_1$ is a closed set.
    \item Iterate this process for each remaining closed cube of $M_n$ to obtain $M_{n + 1}$.
\end{itemize}
Then the \textit{Menger Sponge} is defined as
\begin{equation*}
    M = \displaystyle\bigcap_{n = 0}^{\infty}M_n.
\end{equation*}
\subsection{The Arc Presentation}

As we have mentioned before, every knot has a knot diagram that fits our purposes perfectly. We will review it now. 

An \textit{Arc Presentation in grid form} is encoded in an ordered list of $n$ unordered pairs 
\begin{equation*}
    \{a_1, b_1\},..., \{a_n, b_n\},
\end{equation*}
sucht that $a_1,..., a_n$ and $b_1,..., b_n$ are permutations of $1, ..., n$ and \begin{equation*}
    a_i \neq b_i, i= 1,...., n.
\end{equation*}

We then construct the finite grid of points 
\begin{equation*}
   (a, b), 1\le a, b\le n .
\end{equation*}
It consists of $n^2$ points with integer coordinates. To construct the presentation we do as follows:
\begin{description}
    \item[Step 1:] For each $1\le i\le n$, draw the horizontal line segment joining the points $(a_i, i)$ with $(b_i, i)$.
    \item[Step 2:]  Each vertical line at $j =1, 2, ..., n$ has exactly two points drawn on it. Draw the segment joining those two points for each $j$.
    \item[Step 3:] To resolve the crossings, whenever there is one, always draw the vertical segment \textit{above} the horizontal one.
\end{description}

For a given knot $K$, the minimal $n$ for which the above construction can be carried out is called the \textit{Arc index} of $K$. It is denoted by $\alpha(K)$.

\begin{example}\label{ex: 4_1 arch presentation}
An Arc Presentation of the Eight Knot $4_1$ is
\begin{equation*}
    \{3, 5\}, \{6, 4\}, \{5, 2\}, \{1, 3\}, \{2, 6\}, \{4, 1\}.
\end{equation*}
The obtained knot diagram is 
\begin{center}
    \includegraphics[scale = 0.7]{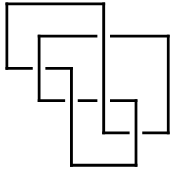}
\end{center}
It can be proven that $\alpha(4_1) = 6.$
\end{example}

We have the following result:

\begin{proposition}\label{prop: arc presentation exists}
Every knot admits an Arc Presentation in grid form. Furthermore, there is an algorithm to construct the Arc Presentation out of any other knot diagram of the knot.
\end{proposition}
\begin{proof}
For a proof that every knot has an Arc Presentation see \cite{CROMWELL199537}. For a discussion of an algorithm to produce grid diagrams, see \cite{AlgorithmKnots}.
\end{proof}

\subsection{Pretzel Knots}

Given a tuple of integers $(q_1,....q_n)$ a \textit{Pretzel knot} (or possibly a link) is the knot (or link) that corresponds to the diagram

\begin{center}
    \includegraphics[scale = 0.5]{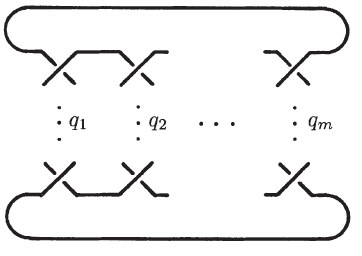}
\end{center}

The sign of $q_i$ corresponds to the orientation of the first crossing of the helix while the magnitude to the number of total crossings in the corresponding helix. That is, for $q_i$ positive the orientation goes as shown while for $q_i$ negative the crossings are flipped.

Notice that not every knot is a Pretzel knot. For example, the first prime knot that is not pretzel is $8_{12}$. We refer the reader to \cite[Theorem 17, Theorem 19]{DIAZ2023108583} for a list of the prime knots up to nine crossings that are or are not Pretzel. 

\section{The positive answer for the Menger Sponge}\label{sec: Positive answer menger}

In this section we will prove that every knot is embedded in a finite iteration of the Menger Sponge. Let us be more precise. For every knot $K$, we will prove that there exists an $n$ such that a closed path on the edges of the boundary of $M_n$ is equivalent to $K$. In this way, since points in the edges are not removed in successive iterations, the knot is also found in the Menger Sponge itself, as desired.

What we do now is as follows: given a knot $K$, we will see that on a face of a large enough iteration of the Menger Sponge we can draw the connectivity graph of an Arc Presentation of $K$. Then, by suitably resolving the intersections, we will be able to \textit{push} the presentation into the sponge in such a way that the vertical and horizontal segments lie in opposite faces and are joined by appropriate paths within the sponge (i.e. we avoid the holes).

The main fact we need for the \textit{pushing} into the Menger Sponge to be possible is given by the following

\begin{lemma}\label{lemma: pushing into menger sponge}
Let $(x, y)$ be a Cantor Dust point, then 
\begin{equation*}
    (x, y, z) \in M
\end{equation*}
for all $0\le z \le 1$.
\end{lemma}
\begin{proof}
For each iteration of the Menger Sponge $M_n$ we define $L_n$ to be the subset of the points $(x, y)$ on the front face of $M_n$ such that $(x, y, z)\in M_n$ for $0\le z\le 1$.

$L_0$ is the whole front face of the cube $M_0$. In the next stage, $L_1$ consists only of the four corner squares as shown in the figure:
\begin{center}
    \includegraphics{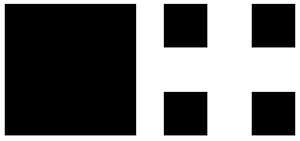}
\end{center}
Every point not shaded black in the front face has some hole of $M_1$ behind it. Notice however that this hole might not share a face with the front face.

In the next stage, this process gets repeated iteratively. It is enough to see what happens in each of the cubes which have a face on the front since those behind are just translations orthogonal to the plane of the front face. Thus, if a hole were to appear in a farther cube there is also one on a closer cube to the front face.

Thus $L_0, L_1, L_2,...$ is obtained by iteratively repeating the process shown in the above figure. We see this is the iterative process of the Cantor Dust of section \ref{sec: review of prereq}. We conclude that the points of the front face that are Cantor Dust points are the ones that can be joined by a straight line to the back face through $M$.
\end{proof}

With this lemma at hand we are ready to give the positive answer to question one from the introduction. Concretely, we have

\begin{theorem}\label{kinsponge}
All knots $K$ can be embedded into a finite iteration of the Menger Sponge.
\end{theorem}
\begin{proof}
For a given knot $K$, proposition \ref{prop: arc presentation exists} above shows that it admits an Arc Presentation
\begin{equation*}
    \{a_1, b_1\},..., \{a_n, b_n\},
\end{equation*}
where $a_1,..., a_n$ and $b_1,..., b_n$ are permutations of $1, 2, ... n$. 

Take an iteration, of the process that produces the Cantor set $C$, that has at least $n$ endpoints. This is possible since the $k^{th}$ iteration has $2^{k+1}$ endpoints. Among those endpoints pick $n$, say, $p_1 < p_2 < ...< p_n$.

We now consider the Arc Presentation codified as
\begin{equation*}
    \{p_{a_1}, p_{b_1}\},..., \{p_{a_n}, p_{b_n}\},
\end{equation*}
where $p_{a_i}$ is considered on the horizontal axis of the front face, while the $p_{b_j}$ is considered on the vertical axis of the front face (with origin in the lower left corner). 

The corresponding diagram has the same knot type. The only difference is that the distance between the vertical segments (or the horizontal segments) changed, since they were shifted to begin at certain endpoints, as opposed to be evenly distributed in $[0, 1]$. Notice that we have not introduced new intersections since we are preserving the order of both horizontal and vertical segments. We say we have shifted the grid diagram.

We now make an observation: if we pick a point on the front face $(x_0, 0, 0)$ such that $x_0$ is a point of the Cantor set, then $(x_0, y, 0)$ is entirely on the front face for $0\le y \le 1$. This follows from proposition \ref{prop: sierpinski carpet characterization} because $x_0$ does not have a $1$ in its ternary expansion, since it is an element of the Cantor set, and so $(x_0, y)$ satisfies the criterion given by the lemma. The analogous result hold for horizontal segments $(x, y_0, 1)$ if $y_0$ is an endpoint of the Cantor set. 

Since the connectivity graph associated to the Arc Presentation given above is constructed by line segments included in segments as those discussed in the previous paragraph, we conclude that all vertical segments are included entirely on the front face. We also conclude that all the horizontal segments can be translated to the back face and they will be included entirely in it.

Our final task is to show that an endpoint of a vertical segment, in the front face, can be joined with the corresponding vertex of the horizontal segment on the back. This is possible, by lemma \ref{lemma: pushing into menger sponge}, because the coordinates of such points are of the form $(x_0, y_0, 0)$ and $(x_0, y_0, 1)$ with $(x_0, y_0)$ in the Cantor Dust. Thus we can go within the sponge avoiding its holes!

With this we have been able to construct a knot inside the Menger Sponge. This knot is $K$ because its knot diagram, when projected on the front face, precisely recovers the Arc Presentation we began with. This concludes the proof.
\end{proof}

\begin{example}\label{ex: 3_1 in the menger}
    The above procedure done for the trefoil knot $3_1$ is as follows. The first image on the left shows the coordinates (just labeled for the $x$ axis) evenly spaced. The shifting described in the proof is shown in the second image. As we can see the vertical and horizontal lines have intersections in Cantor Dust points. 
    
    \begin{center}
        \includegraphics[scale = 0.5]{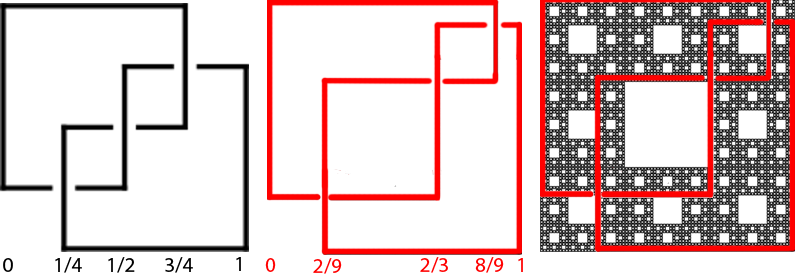}
    \end{center}
    The third image is the connectivity graph in the face of the Menger Sponge.
    
    After pushing we get the following
    image (with and without the Menger Sponge iteration):
    \begin{center}
        \includegraphics[scale = 0.6]{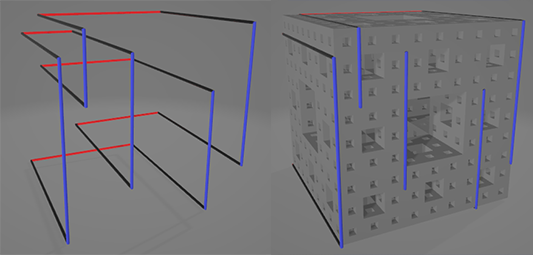}.
    \end{center}
\end{example}
\begin{example}\label{Ex: 4_1 in the Menger}
    Doing the above process with the Arc Presentation of the eight knot $4_1$ that we saw in example \ref{ex: 4_1 arch presentation} we get
    \begin{center}
        \includegraphics[scale = 0.6]{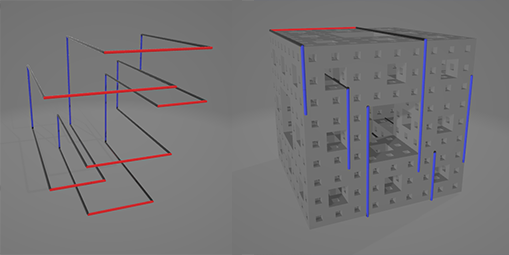}
    \end{center}
In the drawing on the left we have rotated the knot to show how it looks when seen from the back of the sponge.
\end{example}

A well known result that follows from the above is

\begin{corollary}
    The isotopy classes of tame knots form a countable set.
\end{corollary}
\begin{proof}
    The number of finite iterations of the Menger Sponge is countable. The result now follows from theorem \ref{kinsponge}.
\end{proof}

\section{Partial Answer for the Sierpinski Tetrahedron}\label{sec: partial answer sierpinski}

Contrary to what we are able to prove for the Menger Sponge, we do not know if all knots can be embedded into a finite iteration (or the final one) of the Sierpinski Tetrahedron. 

Since we do not have a knot diagram that is particularly suited for the Sierpinski Tetrahedron, we have to develop a different way of searching through it. We now explain how we do this.

\subsection{The Combinatorial Representation of the Tetrahedron}

In order to make the arguments for our proofs, as well as the searches by inspection that we did, it is convenient to flatten the different iterations of the Sierpinski Tetrahedron.

We do so as shown in the following figure for the first and second iterations:
\begin{center}
\includegraphics[scale = 0.05]{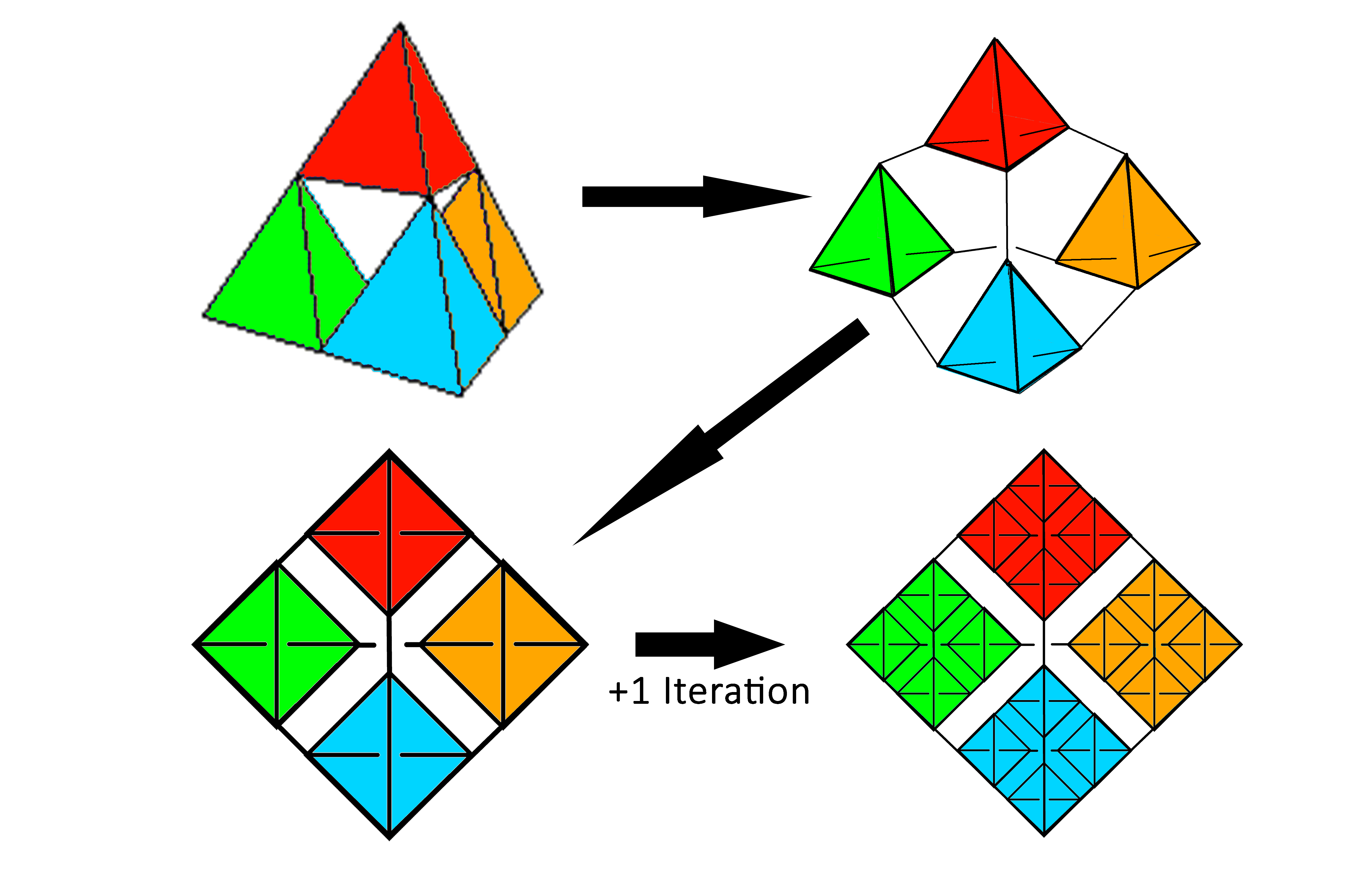}
\end{center}
The final figure now can be used as a map where we look for the planar diagrams of knots. Notice that the graphs that correspond to the successive iterations of the tetrahedron are not planar after the first two. This can be easily proven via Kuratowski's theorem.

\begin{example}\label{ex: example 3_1 sierpinski}
Looking for $3_1$ in the combinatorial representation we get
\begin{center}
    \includegraphics[scale=0.05]{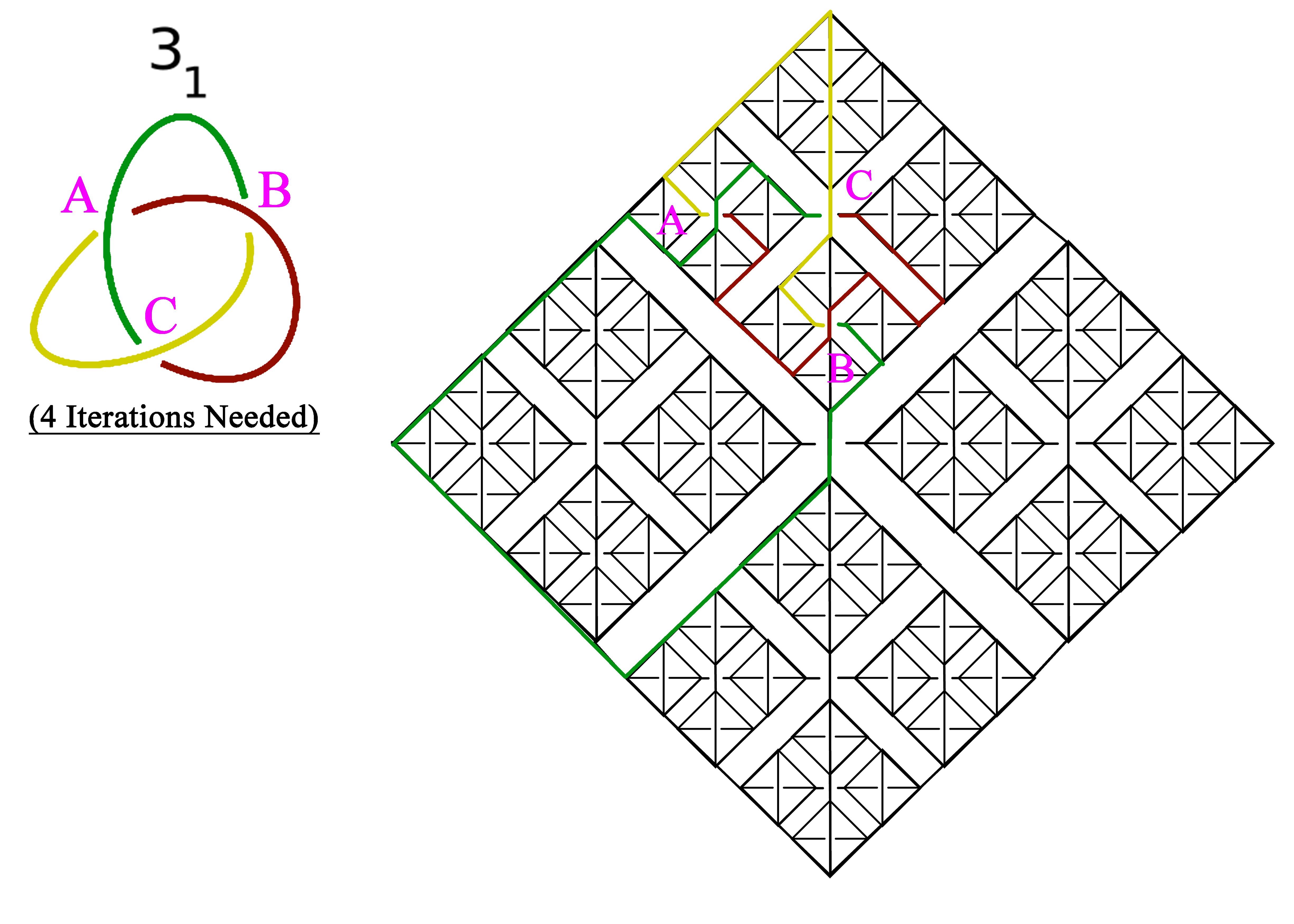}
\end{center}
This representation then becomes the following embedding of $3_1$ in the tetrahedron.
\begin{center}
    \includegraphics[scale = 1.0]{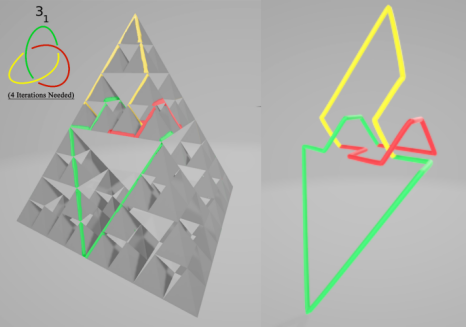}
\end{center}

\end{example}

We emphasize that the combinatorial representation makes the search easier since we avoid the complex nature of the three dimensionality of the tetrahedron. Instead we have this interesting map of how to move in it. The authors were unable to find by inspection many knots in the tetrahedron until we had the combinatorial representation available.

\begin{example}\label{ex: 4_1 in pyramid}
    We obtain in the combinatorial representation the following knot diagram of the eight knot $4_1$.
    \begin{center}
        \includegraphics[scale = 0.5]{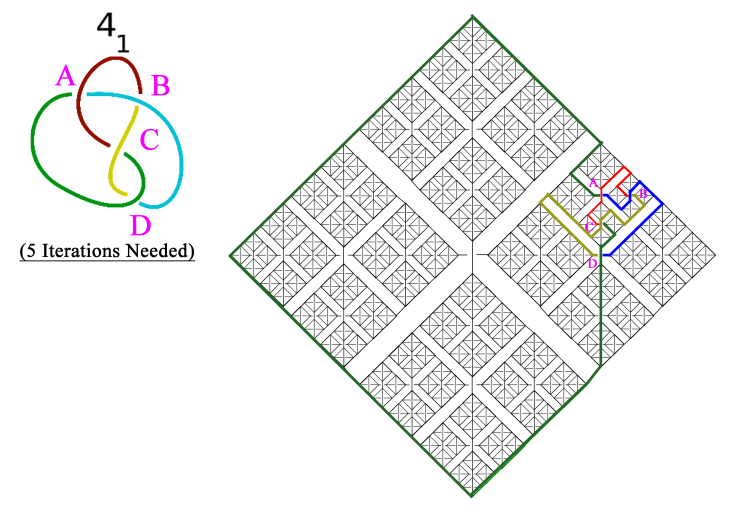}
    \end{center}
    This becomes the following embedding in the tetrahedron.
    \begin{center}
        \includegraphics[]{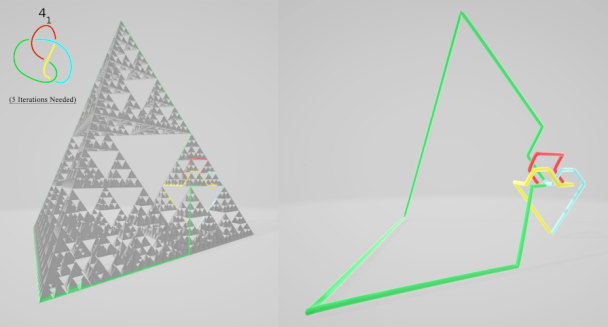}
    \end{center}
    
\end{example}

\subsection{Pretzel Knots}

To prove a given knot $K$ is inside the tetrahedron it would be enough that some knot diagram of $K$ is inside one big enough combinatorial representation of it. We were unable to prove this systematically for all knots but succeeded for Pretzel Knots. 

The results will follow from a series of lemmas, each of which is straightforward from the corresponding picture so we will minimize the discussion to the essentials in their proofs.

\begin{lemma}\label{lemma: independent helices}
Any isolated helix can be put inside a finite iteration of the tetrahedron.
\end{lemma}
\begin{proof}
The images below show that every helix can be found in a big enough iteration of the combinatorial representations. 
\begin{center}
    \includegraphics[scale = 0.5]{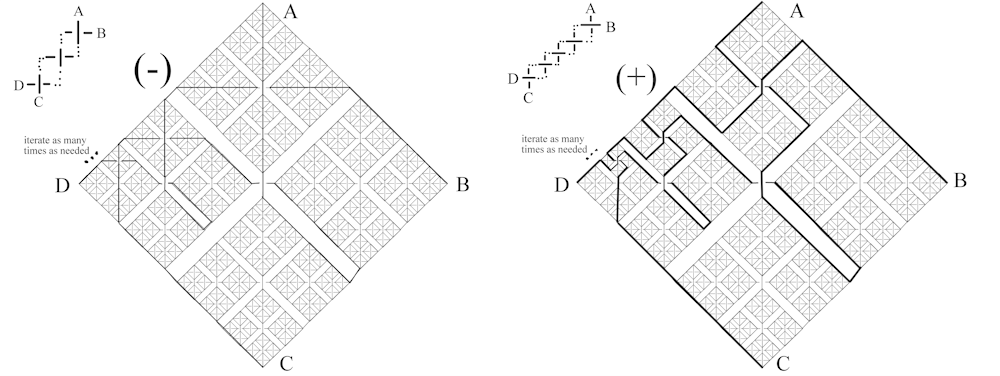}
\end{center}

Notice that what distinguishes the positive crossing from the negative one is which strand comes on top on the first crossing. The one from A for the negative helix and the one from B from the positive helix. This is preserved in the drawings.

Finally, we draw the attention of the reader to the following point: when going from one of the crossings of one helix to the next, we produced the following type of crossing:
\begin{center}
    \includegraphics[]{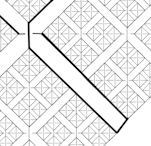}
\end{center}
These crossings are irrelevant, from the perspective of knot type, as they can be removed by a twist (i.e. a Reidemiester move of type I).
\end{proof}

We are now ready to prove
\begin{theorem}\label{thm: pretzel in pyramid}
All Pretzel Knots are inside a finite iteration of the Sierpinski Tetrahedron.
\end{theorem}
\begin{proof}
We have seen in lemma \ref{lemma: independent helices} that individual helices can be put into big enough iterations of the Sierpinski Tetrahedron. We now show that any number of them with any given orientations can be successively joined.

We can simplify a Pretzel Knot diagram to the following simple form
\begin{center}
    \includegraphics[scale = 0.5]{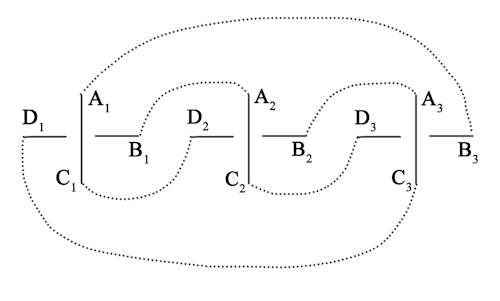}
\end{center}
Each cross represents an helix with $A_i, B_i, C_i, D_i$ representing the different endpoints of the helices as used in lemma \ref{lemma: independent helices}. 

We have proven that each helix is independently found. We now show we can join them as required. The following picture shows how (for concreteness we did the drawing for three helices):
\begin{center}
    \includegraphics[]{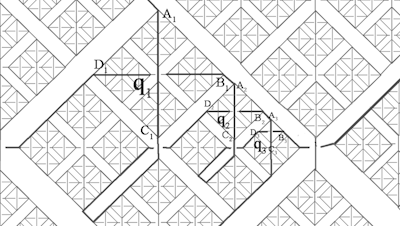}
\end{center}
Notice that $A_1$ gets connected with $B_k$ by a path going through the upper part of the combinatorial representation, while $D_1$ gets connected with $C_k$ by one going through the bottom part of the combinatorial representation.
\end{proof}

\begin{example}
We emphasize that the above process is algorithmic. For example, the knot $8_{15}$ is a Pretzel knot. Concretely, it is $P(-3, 1, 2, 1, -3)$. To find this knot by inspection would be next to impossible, but using the method given in the above proofs and keeping track of the number of iterations needed for all helices to exists as required, we see we can put it inside the eleventh iteration. 
\end{example}

As we have said before, not all knots are Pretzel Knots. Thus, our previous algorithm does not produce the inclusion of every knot in the Sierpinski Tetrahedron. However, the authors do not really see an impediment for other knots to be found there. Thus we put forward the following

\begin{conjecture}
    All knots can be found a finite iteration of the Sierpinski Tetrahedron.
\end{conjecture}

\section{Comparison between Fractals}\label{sec: comparison}

Now that we have discussed separately the Menger Sponge and the Sierpinski Tetrahedron, let us compare the information each of them gives. For this purpose we give the following
\begin{definition}
Let $K$ be a knot. Define $M(K)$ as the minimum number of iterations needed to embed $K$ into $M_n$. Similarly, define $S(K)$ as the minimum number of iterations needed to embed $K$ into $S_n$, supposing such $n$ exists for $K$.
\end{definition}
As we have mentioned in the introduction, the Menger Sponge is a more complicated object than the Sierpinski Tetrahedron. In this direction we will prove below that any knot appears faster in the Menger Sponge than it does in the Sierpinski Tetrahedron.  Everything follows from the next fact.

\begin{lemma}
    The one-skeleton of $S_n$ can be embedded into the one-skeleton of $M_n$ for $n = 0, 1, 2,...$
    Furthermore, such an embedding $i_n$ can be constructed in such a way that it preserves knot type. That is, if $K$ is a knot in the one-skeleton of $S_n$, then $i_n(K)$ is isotopic to $K$ (i.e. the embedding does not produce extra knotting).
\end{lemma}
\begin{proof}
    We proceed by induction. For $M_0$ and $S_0$ the inclusion is as shown in the following picture:
      \begin{center}
           \includegraphics[scale = 0.7]{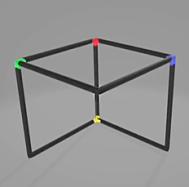}
       \end{center} 
    The four coloured vertices, out of the eight vertices of the cube $M_0$, correspond to those of the tetrahedron $S_0$ under the embedding. 
    For $S_1$ and $M_1$ we have the following picture:
       \begin{center}
           \includegraphics[scale = 0.7]{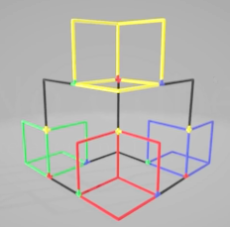}
       \end{center} 
    (Notice how this graph is isomorphic to $S_4^2$ as show in \cite[Figure 1(right), page 567]{SierpinskiGraphs}.)
       
    To produce $S_2$ inside $M_2$ what we do is to substitute each of the coloured $S_0$, inside the $S_1$, by the graph of $S_1$ above. This produces the $S_{2}$. We proceed in this way, inductively, to produce $S_{n+1}$ inside $M_{n+1}$: we substitute each of the $S_0$ in $S_{n+1}$ by $S_1$ above, with the condition that the colours match (i.e. the red corner is where the red cube goes, etc.). For the colours to match one has to swap green and blue in half the cubes, but this doesn't change the graph type.

    Since the graphs we are producing are isomorphic to the skeletons of $S_0, S_1, S_2, S_3,...$ because they replace the last $S_0$ by the corresponding subdivision into $S_1$, which is exactly how we get successive iterations of the Sierpinski Tetrahedron (i.e. subdividing the smallest tetrahedron's, into four smaller ones as explained in section \ref{sec: review of prereq}).

    Notice that these substitutions occur in disjoint regions, so they do not loop among themselves, whence proving that a knot $K$ in $S_n$ is isotypic to $i_n(K)$. Finally, because to add the iteration $S_1$ we only needed one extra subdivision, we indeed have that $S_{n+1}$ lies in $M_{n+1}$. This concludes the proof.
\end{proof} 
We immediately get the following
\begin{theorem}
Let $K$ be a knot, then we have
\begin{equation*}
    M(K) \le S(K),
\end{equation*}
whenever $S(K)$ is defined.
\end{theorem}

\begin{example}
We can verify $S(3_1)\ge 2$. Indeed, the one-skeleton of $S_1$ is a planar graph that is already on the surface of a topological sphere (which is the original tetrahedron). Thus, any knot that is produced by a path on it has genus $0$, and thus trivial. However, by inspection, we can see that $M(3_1) = 1$ (We leave to the interested reader the task of finding the trefoil in the one-skeleton of $M_1$). We thus see the inequality in the previous theorem can be sharp. 
\end{example}

A comparison between the moment of appearance in the Menger Sponge and the Arc index is given by the next corollary. 
\begin{corollary}\label{cor: index vs menger}
Let $K$ be a knot with Arc index $\alpha(K)$, then
\begin{equation*}
    M(K) \le \log_2(\alpha(K)) + 1,
\end{equation*}
or equivalently,
\begin{equation*}
    \alpha(K) \ge 2^{M(K) - 1}.
\end{equation*}
\end{corollary}
\begin{proof}
    This follows immediately from the proof of theorem \ref{kinsponge} above.
\end{proof}

\begin{example}
    The trefoil $3_1$ has $\alpha(3_1) = 5$. The above inequality is $M(3_1)\le 3.322...$. We know $M(3_1) = 1$.
\end{example}

Corollary \ref{cor: index vs menger} compares two different indices constructed out of a knot. On the one hand, we have the Arc index which requires the planarity of the Arc Presentation. Thus, for its construction, it wastes the flexibility that the three dimensionality of a knot gives. On the other hand, $M(K)$ takes advantage of it since the iterations of the Menger Sponge are three dimensional. What the inequality states, conceptually, is that indeed a lot of space is gained by this. However, this inequality was obtained by a very elementary embedding of a knot that still ignores a lot of paths within the iterations of the Menger Sponge (for example, $M_2$ is already very complicated). 

Thus, it is believable that the inequality is always strict (even if we consider integer part of the right hand side), and even that the gap between both sides grows. It is thus an interesting question to know what is the real asymptotic behaviour of $M(K)$.

\section{Final Question}\label{sec: Final Comments}

We conclude this paper with a question which we decided to emphasize by isolating it to its own section. We have seen that all knots exists within the iterations of the Menger Sponge, and that certain ones lie in the Sierpinski Tetrahedron. As we mentioned, we suspect actually all knots should be there as well. This leads to the our question of which we have no answer: \textit{is there a fractal, produced by an iterative process, that admits certain types of knots but not others? If so, what does that say of the fractal itself or of the families it avoids.} An example of such constructions would be quite impressive from our point of view.


\bibliographystyle{acm}
\bibliography{Bibliography}

\begin{thebibliography}{1}

\bibitem{Allouche_Shallit_2003}
{\sc Allouche, J.-P., and Shallit, J.}
\newblock {\em Automatic Sequences: Theory, Applications, Generalizations}.
\newblock Cambridge University Press, 2003.

\bibitem{CROMWELL199537}
{\sc Cromwell, P.~R.}
\newblock Embedding knots and links in an open book i: Basic properties.
\newblock {\em Topology and its Applications 64}, 1 (1995), 37--58.

\bibitem{DIAZ2023108583}
{\sc Díaz, R., and Manchón, P.}
\newblock Pretzel knots up to nine crossings.
\newblock {\em Topology and its Applications 339\/} (2023), 108583.

\bibitem{KnotPartitionFractals}
{\sc Ge, M.-L., Hu, L., and Wang, Y.}
\newblock Knot theory, partition function and fractals.
\newblock {\em Journal of Knot Theory and Its Ramifications 05}, 01 (1996), 37--54.

\bibitem{AlgorithmKnots}
{\sc Gyo Taek~Jin, H. J.~L.}
\newblock Minimal grid diagrams of 11 crossing prime alternating knots.
\newblock {\em Journal of Knot Theory and Its Ramifications 29}, 11 (2020).

\bibitem{SierpinskiGraphs}
{\sc Hinz, A.~M., Klavžar, S., and Zemljič, S.~S.}
\newblock A survey and classification of {S}ierpiński-type graphs.
\newblock {\em Discrete Applied Mathematics 217\/} (2017), 565--600.

\bibitem{KnotsnGraphs}
{\sc Kurpita, B., and Murasugi, K.}
\newblock Knots and graphs.
\newblock {\em Chaos, Solitons and Fractals 9}, 4 (1998), 623--643.
\newblock Knot Theory and Its Applications: Expository Articles on Current Research.

\bibitem{chaosandfractals}
{\sc Peitgen, H.}
\newblock {\em Chaos and Fractals New Frontiers of Science - Second Edition}.
\newblock Springer-Verlag, New York, 2004.

\end{thebibliography}

\end{document}